\documentclass[11pt,reqno]{amsart}

\usepackage{amsmath,amssymb,amsthm,mathtools,mathrsfs}
\usepackage{enumitem}
\usepackage{needspace}
\usepackage{microtype}
\usepackage[hidelinks]{hyperref}
\usepackage[nameinlink,capitalise]{cleveref}
\hypersetup{pdftitle={Moments of the Cross-Sectional Siegel--Veech Transform},
 pdfauthor={Albert Artiles}}

\newcommand{\R}{\mathbb R}
\newcommand{\Z}{\mathbb Z}
\newcommand{\N}{\mathbb N}
\newcommand{\G}{\mathrm{SL}(2,\R)}
\newcommand{\one}{\mathbf 1}
\newcommand{\prim}{\mathrm{prim}}
\newcommand{\Hol}{\operatorname{Hol}}
\newcommand{\supp}{\operatorname{supp}}

\newcommand{\detm}{\operatorname{det}}

\newcommand{\cF}{\mathcal F}

\newcommand{\cA}{\mathcal A}

\newcommand{\Om}{\Omega}
\newcommand{\Lam}{\Lambda}
\newcommand{\Gam}{\Gamma}
\newcommand{\ph}{\varphi}

\newcommand{\Id}{\operatorname{Id}}

\theoremstyle{plain}
\newtheorem{theorem}{Theorem}[section]

\newtheorem{lemma}[theorem]{Lemma}
\newtheorem{corollary}[theorem]{Corollary}

\theoremstyle{definition}

\theoremstyle{remark}
\newtheorem{remark}[theorem]{Remark}

\title[Moments of the Cross-Sectional Siegel--Veech Transform]
{Moments of the Cross-Sectional Siegel--Veech Transform}

\author{Albert Artiles}
\address{Yau Mathematical Sciences Center, Tsinghua University, Beijing, China}
\email{artilesa@mail.tsinghua.edu.cn}


\begin{document}

\begin{abstract}
We study the Siegel--Veech transform on the Poincar\'e section for the horocycle flow consisting of lattice surfaces with a visible horizontal holonomy vector of length at most one. We compute its first moment and derive formulas for higher and factorial moments with respect to the section's induced ergodic invariant probability measure, which is not supported on a periodic orbit. We emphasize these formulas in the case of the square torus of unit area and, as an application, use its first moment formula to recover the Boca--Zaharescu pair-correlation density for Farey fractions.

\end{abstract}

\maketitle

\section{Introduction}

The Siegel transform is one of the basic tools in the geometry of numbers. Given an integer $d\geq2$ and a bounded compactly supported Borel function on $\R^d$, one obtains a function on the space of unimodular lattices by summing over the nonzero lattice vectors. Siegel's mean value theorem \cite{Siegel1945} tells us that the first moment of this transform can be obtained by integrating the original function with respect to Lebesgue measure. Rogers developed higher-moment formulas \cite{Rogers1955}, and Schmidt obtained related metrical counting results \cite{Schmidt1960}. 
These identities turn questions about the statistics of lattice vectors into integration problems on homogeneous spaces.

For translation surfaces, the analogous construction is the Siegel--Veech transform. A translation surface carries discrete collections of holonomy vectors coming from saddle connections and cylinders. Masur proved quadratic bounds for the number of saddle connections of length at most $R$ \cite{Masur1988,Masur1990}. Veech placed these counting problems in an ergodic framework and established the Siegel--Veech mean value formula for equivariant closed discrete sets \cite{Veech1998}. Eskin and Masur used this framework to obtain asymptotic counting formulas for flat surfaces \cite{EskinMasur2001}. The measure-classification theorems of Eskin--Mirzakhani and Eskin--Mirzakhani--Mohammadi later described the invariant measures and orbit closures that arise naturally in these averaging problems \cite{EskinMirzakhani2018,EskinMirzakhaniMohammadi2015}.

Higher moments record correlations among several holonomy vectors. Athreya--Cheung--Masur proved that Siegel--Veech transforms of bounded compactly supported functions belong to $L^2$ with respect to Masur--Veech measure on connected components of area-one strata \cite{AthreyaCheungMasur2019}. For lattice orbits, Fairchild computed higher moments of Siegel--Veech transforms for Hecke triangle groups \cite{Fairchild2021}. More recently, Burrin and Fairchild proved a two-orbit Siegel--Veech formula for general discrete lattice orbits in $\R^2$ and applied it to pairs of saddle connections on Veech surfaces \cite{BurrinFairchild2024}. The formulas of Siegel, Rogers, Fairchild, and Burrin--Fairchild average over the full homogeneous space. The purpose of the present paper is to study the corresponding moment problem after restricting the Siegel--Veech transform to a Poincar\'e section for the horocycle flow.

\subsection*{The Poincar\'e section and the cross-sectional transform}


Let $(X,\omega)$ be a lattice surface with Veech group $\Gamma<\G$, and let $\Lambda\subset\R^2\setminus\{0\}$ be the set of visible holonomy vectors of $(X,\omega)$. We identify the $\G$-orbit of $(X,\omega)$ with $\G/\Gam$.

The Poincar\'e section used throughout this paper is the set of translation surfaces in $\G/\Gamma$ having a visible horizontal holonomy vector of length at most one,
\begin{equation}\label{eq:intro-section-abstract}
 \mathcal L
 =\left\{g\Gam\in\G/\Gam:
 g\Lam\cap\bigl((0,1]\times\{0\}\bigr)\neq\varnothing\right\}.
\end{equation}
For $\Gam=\mathrm{SL}(2,\Z)$, Athreya and Cheung showed that $\mathcal{L}$ is a Poincar\'e section for the horocycle flow and that its first return map is equivalent to the Boca--Cobeli--Zaharescu (BCZ) map \cite{AthreyaCheung2014}. The BCZ map was first introduced by Boca, Cobeli, and Zaharescu in their study of the statistics of Farey fractions \cite{BocaCobeliZaharescu2001}. Athreya--Chaika--Leli\`evre studied the analogous section for the golden $L$ \cite{AthreyaChaikaLelievre2015}. Uyanik--Work parameterized the section for Veech surfaces with finitely many cusps \cite{UyanikWork2016}, generalizing the constructions of Athreya--Cheung and Athreya--Chaika--Leli\`evre. The author later used this section to prove weak mixing of the return map for lattice surfaces \cite{ArtilesWeakMixing}, extending the weak-mixing theorem of Cheung and Quas for the BCZ map \cite{CheungQuas2024}.

For simplicity, we assume in the main statements that $\G/\Gam$ has one cusp and that $-\Id\in\Gam$. We discuss the other parabolic normal forms and the case of several cusps in \cref{rem:minus-id,rem:several-cusps}. After applying an element of $\G$ to $\omega$, we may assume that $e_1=(1,0)^T$ is a visible holonomy vector. Let $a\in\R$, and write
\[
 h_a=\begin{pmatrix}1&a\\0&1\end{pmatrix}.
\]
The \emph{cusp width} is
\[
 \alpha
 =\min\{a>0:h_a\in\Gam\text{ and }h_ae_1=e_1\}.
\]
Equivalently, $h_\alpha$ generates the vector stabilizer $\Gam_{e_1}=\{g\in\Gam:ge_1=e_1\}$.

For $s\neq0$ and $t\in\R$, set
\[
 p_{s,t}=\begin{pmatrix}s&t\\0&s^{-1}\end{pmatrix}.
\]
The section coordinates of Athreya--Chaika--Leli\`evre \cite{AthreyaChaikaLelievre2015}, in the general form given by Uyanik--Work \cite{UyanikWork2016}, parameterize $\mathcal L$ by
\begin{equation}\label{eq:intro-section-triangle}
 \Om_\alpha
 =\{(s,t)\in\R^2:0<s\leq1,\ 1-\alpha s<t\leq1\},
 \qquad
 (s,t)\longleftrightarrow p_{s,t}\Gam,
\end{equation}
with invariant probability measure
\begin{equation}\label{eq:intro-section-measure}
 dm_\alpha(s,t)=\frac{2}{\alpha}\,ds\,dt.
\end{equation}

Denote the set of visible holonomy vectors of $\omega$ with positive height by $\Lam_+$, that is,
\[
 \Lam_+=\{v\in\Lam:\pi_2(v)>0\}.
\]
For a bounded compactly supported Borel function $f$ on $\R\times(0,\infty)$, we define the positive cross-sectional Siegel--Veech transform by
\begin{equation}\label{eq:intro-positive-transform}
 \widehat f_+(s,t)
 =\sum_{v\in\Lam_+}f(p_{s,t}v).
\end{equation}
We collect the heights of holonomy vectors of $(X,\omega)$ in the set
\begin{equation}\label{eq:intro-J}
 J=\pi_2(\Lam)\cap(0,\infty)
\end{equation}
and, extending the geometric Euler totient function introduced for Hecke triangle groups in \cite[Definition~1.4]{Fairchild2021}, define
\begin{equation}\label{eq:intro-phi}
 \ph(\zeta)
 =\#\{x\in\R:(x,\zeta)\in\Lam,\ 0\leq x<\alpha\zeta\},
 \qquad \zeta\in J.
\end{equation}

We refer to the quantities $\alpha$, $J$, and $\varphi$ as the cusp data of $(X,\omega)$.

Consider the cumulative sum of the geometric Euler totient function, which we denote by
\begin{equation}\label{eq:intro-cumulative}
 \Phi_\omega(y)
 =\sum_{\substack{\zeta\in J\\\zeta\leq y}}\ph(\zeta).
\end{equation}

Our first main result computes the first cross-sectional moment for a Veech surface.

\begin{theorem}\label{thm:intro-main-first}
Let $(X,\omega)$ be as above. Then for every bounded, compactly supported Borel function $f:\R\times(0,\infty)\to\R$,
\begin{equation}\label{eq:intro-main-first}
 \int_{\Om_\alpha}\widehat f_+\,dm_\alpha
 =\frac{2}{\alpha}
 \sum_{\zeta\in J}\ph(\zeta)
 \int_\zeta^\infty\int_\R
 f(x,y)\,\frac{1}{y^2}\,dx\,dy.
\end{equation}
Equivalently,
\begin{equation}\label{eq:intro-main-first-density}
 \int_{\Om_\alpha}\widehat f_+\,dm_\alpha
 =\frac{2}{\alpha}
 \int_0^\infty\int_\R
 f(x,y)\frac{\Phi_\omega(y)}{y^2}\,dx\,dy.
\end{equation}
\end{theorem}

The density in \eqref{eq:intro-main-first-density} depends only on the height $y$. Its value is determined by the cusp data through the sum $\Phi_\omega(y)$. This differs from the classical Siegel--Veech formula, where the density is constant. 

\subsection*{The modular case}

Let
\[
 \Gam=\mathrm{SL}(2,\Z) 
 \qquad \text{ and }
 \Lam=\Z^2_{\prim}.
\]
Here $\Gamma$ is the Veech group of the square torus with one marked point, and $\Z_{\prim}^2$ is its set of holonomy vectors.

Then $\alpha=1$ and the section is the BCZ triangle
\begin{equation}\label{eq:intro-bcz-triangle}
 \Om
 =\{(s,t)\in\R^2:0<s\leq1,\ 0<t\leq1,\ s+t>1\},
 \qquad dm=2\,ds\,dt.
\end{equation}
We write the set of primitive integer vectors with positive height as
\[
 \Z^2_{\prim,+}
 =\{(A,j)\in\Z^2:j\geq1,\ \gcd(A,j)=1\}.
\]

Using the notation above, we have
\[
 J=\N,
 \qquad
 \ph(j)=\phi(j),
\]
where $\phi$ is the classical Euler totient function from number theory. \Cref{thm:intro-main-first} therefore gives the following formula.

\begin{theorem}\label{thm:intro-main-modular-first}
Let $\Gam=\mathrm{SL}(2,\Z)$, $\Lam=\Z^2_{\prim}$, and let $\Om$ be the BCZ triangle with probability measure $dm=2\,ds\,dt$. Then, for every bounded compactly supported Borel function $f:\R\times(0,\infty)\to\R$,
\begin{equation}\label{eq:intro-main-modular-first}
 \int_\Om\widehat f_+\,dm
 =2\sum_{j\geq1}\phi(j)
 \int_j^\infty\int_\R
 f(x,y)\,\frac{1}{y^2}\,dx\,dy.
\end{equation}
Equivalently,
\begin{equation}\label{eq:intro-main-modular-density}
 \int_\Om\widehat f_+\,dm
 =2\int_0^\infty\int_\R
 f(x,y)\frac{1}{y^2}
 \left(\sum_{1\leq j\leq y}\phi(j)\right)dx\,dy.
\end{equation}
\end{theorem}

Our third main result is the second moment on the modular section. For a bounded compactly supported Borel function $F$ on $(\R\times(0,\infty))^2$, define
\begin{equation}\label{eq:intro-pair-transform}
 \widehat F^{(2)}(s,t)
 =\sum_{v,w\in\Z^2_{\prim,+}}F(p_{s,t}v,p_{s,t}w).
\end{equation}
For $j,k\in\N$ and $n\in\Z$, define
\begin{equation}\label{eq:intro-Phi-jkn}
 \Phi_{j,k}(n)
 =\#\left\{(a,b)\in(\Z/j\Z)^\times\times(\Z/k\Z)^\times:
 ak-bj\equiv n\pmod{jk}\right\}.
\end{equation}
When $j=1$ or $k=1$, the corresponding unit group is understood to contain its unique residue class.

\begin{theorem}\label{thm:intro-main-second}
Let $(X,\omega)$ be the unit area square torus. For every bounded compactly supported Borel function $F:(\R\times(0,\infty))^2\to\R$,
\begin{align}\label{eq:intro-main-second}
 \int_\Om\widehat F^{(2)}\,dm
 &=2\sum_{j,k\geq1}\sum_{n\in\Z}
 \Phi_{j,k}(n)
 \int_j^\infty\int_\R
 F\left(
 (x,y),
 \left(\frac{k}{j}x-\frac{n}{y},\frac{k}{j}y\right)
 \right)\frac{1}{y^2}\,dx\,dy.
\end{align}

\end{theorem}

In \cref{sec:higher}, we extend the same argument to higher moments for one-cusp Veech surfaces. We also compute factorial moments by restricting to tuples of pairwise distinct vectors.

The modular formulas also connect the cross-sectional transform to Farey fractions. The BCZ map is the first return map of the horocycle flow to \eqref{eq:intro-section-abstract}. This correspondence allows us to use the first moment in \cref{thm:intro-main-modular-first} to recover the classical Boca--Zaharescu Farey pair-correlation function.


\section{Preliminaries}\label{sec:prelim}

We fix notation and recall the background used throughout the paper.

\subsection{Translation surfaces and lattice surfaces}

A translation surface is a pair $(X,\omega)$, where $X$ is a compact Riemann surface and $\omega$ is a nonzero holomorphic $1$-form on $X$. We often write $\omega$ for the translation surface and keep $X$ implicit. A convenient geometric way to picture a translation surface is as a finite disjoint union of polygons
\[
P=\bigsqcup_{i=1}^m P_i\subset\mathbb C\simeq\R^2
\]
whose sides are partitioned into pairs and identified by Euclidean translations. Gluing paired sides produces a closed surface, and the form $dz$ descends to a holomorphic $1$-form. Conversely, every translation surface admits a polygonal representation of this form. We refer to Zorich \cite{Zorich2006} for a detailed survey.

Let $\Sigma$ denote the set of zeros of $\omega$ in $X$. The form $\omega$ induces a flat metric on $X\setminus\Sigma$, with singularities at the points of $\Sigma$. For a torus, we include a marked regular point in $\Sigma$ so that saddle connections are defined in the same way. A saddle connection is a geodesic segment joining two points of $\Sigma$, not necessarily distinct, whose interior contains no point of $\Sigma$. If $\gamma$ is an oriented saddle connection, its holonomy vector is
\[
 \Hol_\omega(\gamma)=\int_\gamma\omega\in\mathbb C\simeq\R^2.
\]
Let $H_\omega$ be the set of holonomy vectors of all oriented saddle connections on $\omega$. This set is closed and discrete in $\R^2$, and $-H_\omega=H_\omega$.

There is a natural action of $\G$ on translation surfaces. If $\omega$ is represented by polygons $P\subset\R^2$ and $g\in \G$, then $gP$ defines a new translation surface denoted by $g\omega$. Holonomy vectors transform equivariantly:
\begin{equation}\label{eq:equivariance}
 H_{g\omega}=gH_\omega.
\end{equation}

For the cross section we use visible holonomy vectors. Define
\[
\Lam_\omega
 =\left\{v\in H_\omega:
 tv\notin H_\omega\text{ for every }0<t<1\right\}.
\]

In the case of the square torus with one marked point, the set of visible holonomy vectors is the full set of holonomy vectors and coincides with pairs of coprime integers:
\[
\Z^2_{\prim}=\{(a,b)\in\Z^2:\gcd(a,b)=1\}.
\]
We write $\Lam$ for the visible holonomy set when the surface is understood. It is closed, discrete, and centrally symmetric. Our transforms count each vector once, even if several saddle connections have the same holonomy vector.

The stabilizer of $\omega$ under the $\G$-action is called its Veech group and is denoted by $\Gam_\omega$.

If $\Gam_\omega$ is a lattice in $\G$, then $\omega$ is called a lattice surface or a Veech surface. The $\G$-orbit of $\omega$ is identified with the finite-volume homogeneous space
\[
 \G\omega\simeq \G/\Gam_\omega,
 \qquad \text{ via }g\omega\longleftrightarrow g\Gam_\omega.
\]
Veech \cite{Veech1989,Veech1998} showed that for a lattice surface, the set of visible holonomy vectors is a finite union of $\Gam$-orbits associated to the cusps of $\G/\Gam$.

Let $N(\omega,R)$ be the number of saddle connections of length at most $R$. Masur's lower bound \cite{Masur1988} and upper bound \cite{Masur1990} together give constants $c_1(\omega),c_2(\omega)>0$ such that
\[
 c_1(\omega)R^2\leq N(\omega,R)\leq c_2(\omega)R^2
\]
for all sufficiently large $R$. Veech obtained precise quadratic asymptotics for lattice surfaces \cite{Veech1989,Veech1998}. Eskin and Masur proved an almost-everywhere quadratic asymptotic with respect to Masur--Veech measure on each connected component of an area-one stratum \cite{EskinMasur2001}. Eskin--Masur--Zorich computed the constants for configurations of saddle connections and cylinders in terms of volumes of strata \cite{EskinMasurZorich2003}.


\subsection{The horocycle Poincar\'e section}\label{subsec:section}

Consider the one-parameter subgroup $\{u_r:r\in\R\}$ of $\G$, where
\[
 u_r=\begin{pmatrix}1&0\\-r&1\end{pmatrix},
 \qquad r\in\R.
\]
The elements $u_r$ act on $\G/\Gam$ by left multiplication. This action is called the horocycle flow.

For $h>0$, define the set of $h$-horizontally short surfaces by
\[
 \mathcal L_h
 =\left\{g\Gam\in \G/\Gam:
 g\Lam\cap\bigl((0,h]\times\{0\}\bigr)\neq\varnothing\right\}.
\]
We write $\mathcal L=\mathcal L_1$. For the modular surface, Athreya and Cheung proved that $\mathcal L$ is a Poincar\'e section for the horocycle flow and identified its return map with the BCZ map \cite{AthreyaCheung2014}. The same viewpoint was used by Athreya--Chaika--Leli\`evre in their study of the slope gap distribution of holonomy vectors of the golden $L$ \cite{AthreyaChaikaLelievre2015}. Uyanik--Work extended the construction to describe the corresponding section for an arbitrary lattice surface as a finite disjoint union of triangular regions \cite{UyanikWork2016} and provided an algorithm to compute the return time of the horocycle flow back to $\mathcal{L}$. Kumanduri--Sanchez--Wang refined the choice of section coordinates so that the computation uses only finitely many contributing saddle connections \cite{KumanduriSanchezWang2024}. They proved that the slope gap distribution has finitely many points of non-analyticity and that its tail has quadratic decay. Osman--Southerland--Wang proved effective equidistribution for intersections of long horocycles with these sections and used it to obtain effective slope gap distributions \cite{OsmanSoutherlandWang2025}. The author used the Uyanik--Work description of $\mathcal{L}$ and bounds on the cross-sectional Siegel--Veech transform along a one-parameter family of functions to prove that the return map for lattice surfaces is weakly mixing \cite{ArtilesWeakMixing}.

We recall the coordinates used to parameterize $\mathcal L$. Suppose first that $\G/\Gam$ has one cusp and $-\Id\in\Gam$. After acting on $\omega$ by some element of $\G$, we may suppose that $(1,0)^T\in \Lambda_\omega$. Then there exists $a>0$ such that
\[
 h_a=\begin{pmatrix}1&a\\0&1\end{pmatrix}\in\Gamma.
\]
We define the cusp width by
\[
 \alpha
 =\min\{a>0:h_a\in\Gam\text{ and }h_ae_1=e_1\}.
\]
The vector stabilizer $\Gam_{e_1}=\{g\in\Gam:ge_1=e_1\}$ is generated by $h_\alpha$. Consider the upper triangular matrices
\[
 p_{s,t}=\begin{pmatrix}s&t\\0&s^{-1}\end{pmatrix}.
\]
$\mathcal{L}$ is parameterized by
\begin{equation}\label{eq:section-triangle}
 \Om_\alpha
 =\{(s,t)\in\R^2:0<s\leq 1,\ 1-\alpha s<t\leq 1\},
 \qquad \text{ via }
 (s,t)\longleftrightarrow p_{s,t}\Gam.
\end{equation}
In these coordinates, the induced invariant probability measure $m_\alpha$ is given by the normalized Lebesgue measure
\begin{equation}\label{eq:section-measure}
 dm_\alpha(s,t)=\frac{2}{\alpha}\,ds\,dt.
\end{equation}

When $\Gam=\mathrm{SL}(2,\Z)$, we have $\alpha=1$ and recover the BCZ triangle
\begin{equation}\label{eq:bcz-triangle}
 \Om=\{(s,t):0<s\leq1,\ 0<t\leq1,\ s+t>1\},
 \qquad dm=2\,ds\,dt.
\end{equation}

If $-\Id\notin\Gam$, the coordinates also depend on the eigenvalue of a primitive generator of the stabilizer of the horizontal line. Uyanik--Work describe the two possibilities in \cite{UyanikWork2016}. We explain how the first moment formula applies in both cases in \cref{rem:minus-id}.

\subsection{The classical Siegel--Veech transform}\label{subsec:SV}

Let $\omega$ be a lattice surface with Veech group $\Gamma$. For $f\in C_c(\R^2)$, the Siegel--Veech transform of $f$ is given by
\begin{equation}\label{eq:SV-transform}
 \overline f(g\Gam)=\sum_{v\in\Lam}f(gv).
\end{equation}

Veech's formula states that for the $\G$-invariant probability measure $\mu$ on $\G/\Gamma$, there is a constant $c_{\omega}>0$ such that
\begin{equation}\label{eq:classical-SV}
 \int_{\G/\Gamma} \overline f\,d\mu
 =c_{\omega}\int_{\R^2}f(x)\,dx.
\end{equation}

The main feature of \eqref{eq:classical-SV} is that the first moment measure is a constant multiple of Lebesgue measure. For unimodular lattices, Siegel's mean value theorem gives constant $1$ when summing over all nonzero lattice vectors, and constant $1/\zeta(2)=6/\pi^2$ when summing over primitive vectors \cite{Siegel1945}.

Higher moments ask for averages of products of transforms, or equivalently sums over tuples of vectors. Rogers developed higher-moment formulas on spaces of unimodular lattices \cite{Rogers1955}. Fairchild obtained explicit higher-moment formulas for Hecke triangle groups \cite{Fairchild2021}. Burrin--Fairchild computed a two-orbit Siegel--Veech formula for general lattice orbits and used it to study pairs of saddle connections on Veech surfaces \cite{BurrinFairchild2024}. Our formulas differ in that the averaging measure is the induced measure on the horocycle section rather than Haar measure on the full homogeneous space.

Since the cross-sectional coordinates preserve the sign of the second coordinate when $s>0$, it is convenient to work with the collection of visible holonomy vectors with positive height. For a bounded compactly supported Borel function on the upper half-plane, define
\begin{equation}\label{eq:positive-transform}
 \widehat f_+(p_{s,t}\Gam)
 =\sum_{v\in\Lam\, :\, \pi_2(v)>0}f(p_{s,t}v).
\end{equation}
At each point of the section, the sum is finite because $p_{s,t}\Lam$ is closed and discrete and $f$ has compact support.

\subsection{Moments and factorial moments}\label{subsec:moments}

For $r\in\N$ and a bounded compactly supported Borel function $F$ on $(\R\times(0,\infty))^r$, define
\[
 \widehat F^{(r)}(s,t)
 =\sum_{v_1,\ldots,v_r\in\Lam_+}
 F(p_{s,t}v_1,\ldots,p_{s,t}v_r).
\]
If $F(v_1,\ldots,v_r)=f_1(v_1)\cdots f_r(v_r)$, then
\[
 \widehat F^{(r)}(s,t)
 =\widehat{f_1}_+(s,t)\cdots\widehat{f_r}_+(s,t).
\]
In particular, taking every $f_i=f$ gives $\widehat F^{(r)}=\widehat f_+^{\,r}$. Thus integrating these tuple sums includes the usual moments of the transform.

For factorial moments, we require the vectors to be pairwise distinct. We write
\begin{equation}\label{eq:factorial-transform-def}
 \widehat F^{[r]}(s,t)
 =\sum_{v_1,\ldots,v_r\in\Lam_+}^{\neq}
 F(p_{s,t}v_1,\ldots,p_{s,t}v_r),
\end{equation}
where the superscript $\neq$ indicates this restriction. To see the reason for the terminology, let $A$ be a compact subset of $\R\times(0,\infty)$ and set
\[
 N_A(s,t)=\#\{v\in\Lam_+:p_{s,t}v\in A\}.
\]
For $F=\one_{A^r}$, the ordinary transform gives $N_A^r$, while the factorial transform gives
\[
 (N_A)_r=N_A(N_A-1)\cdots(N_A-r+1),
\]
the falling factorial of $N_A$. Its integral is the corresponding factorial moment. 

\section{The first cross-sectional moment}\label{sec:first}

We begin with a lattice surface whose Veech group has one cusp. Throughout \cref{sec:first,sec:second,sec:higher}, unless stated otherwise, we use the normalization of \cref{subsec:section} and assume $-\Id\in\Gam$.

\subsection{Cusp data}

Let
\begin{equation}\label{eq:J-def}
 J=\pi_2(\Lam)\cap(0,\infty).
\end{equation}
The set $J$ has only finitely many elements in each bounded interval. Indeed, if $0<\zeta\leq Y$, the relation $h_\alpha\Lam=\Lam$ lets us choose a vector at height $\zeta$ whose first coordinate lies in $[0,\alpha\zeta)$. All such representatives lie in the bounded rectangle $[0,\alpha Y]\times[0,Y]$. Since the holonomy set is closed and discrete in $\R^2$, there are only finitely many of them in $[0,\alpha Y]\times [0,Y]$. In particular, $J$ has a least element, which is positive.

For $\zeta\in J$, we denote
$\Lam_\zeta
 :=\Lam\cap(\R\times\{\zeta\}).$
The geometric Euler totient function is given by
\begin{equation}\label{eq:phi-general}
 \ph(\zeta)
 =\#\{x\in\R:(x,\zeta)\in\Lam,\ 0\leq x<\alpha\zeta\}.
\end{equation}
Let
\begin{equation}\label{eq:reps}
 x_{\zeta,1},\ldots,x_{\zeta,\ph(\zeta)}\in[0,\alpha\zeta)
\end{equation}
be the first coordinates of these representatives. Then
\begin{equation}\label{eq:height-decomp}
 \Lam_\zeta
 =\left\{(x_{\zeta,a}+q\alpha\zeta,\zeta):
 1\leq a\leq\ph(\zeta),\ q\in\Z\right\}.
\end{equation}

We define the cumulative geometric Euler totient function as
\begin{equation}\label{eq:Phi-omega}
 \Phi_\omega(y)=\sum_{\substack{\zeta\in J\\ \zeta\leq y}}\ph(\zeta).
\end{equation}

\subsection{The first moment formula}

\begin{theorem}\label{thm:first-general}
Let $\omega$ be a lattice surface with Veech group $\Gam<\mathrm{SL}(2,\R)$. Assume that $\mathrm{SL}(2,\R)/\Gam$ has one cusp and that $-\Id\in\Gam$. Then, for every bounded compactly supported Borel function $f:\R\times(0,\infty)\to\R$,
\begin{equation}\label{eq:first-general}
 \int_{\Om_\alpha}\widehat f_+\,dm_\alpha
 =\frac{2}{\alpha}
 \sum_{\zeta\in J}\ph(\zeta)
 \int_\zeta^\infty\int_\R
 f(x,y)\,\frac{1}{y^2}\,dx\,dy.
\end{equation}
Equivalently,
\begin{equation}\label{eq:first-density}
 \int_{\Om_\alpha}\widehat f_+\,dm_\alpha
 =\frac{2}{\alpha}
 \int_0^\infty\int_\R
 f(x,y)\frac{\Phi_\omega(y)}{y^2}\,dx\,dy.
\end{equation}
\end{theorem}

\begin{proof}
First suppose that $f$ is a nonnegative function. Fix a height $\zeta\in J$ and one representative $x_{\zeta,a}$ from \eqref{eq:reps}.

Set
\[
 v_q=(x_{\zeta,a}+q\alpha\zeta,\zeta),
 \qquad q\in\Z.
\]
Applying $p_{s,t}$ gives
\begin{equation}\label{eq:first-action}
 p_{s,t}v_q
 =\left(sx_{\zeta,a}+t\zeta+q\alpha\zeta s,\frac{\zeta}{s}\right).
\end{equation}
Using the parameterization of $\mathcal L$ and the measure in \eqref{eq:section-measure}, we first integrate the contribution of this family in the $t$-variable. Fix $s$, $\zeta$, and $a$. For each $q\in\Z$, set
\[
 u=sx_{\zeta,a}+t\zeta+q\alpha\zeta s.
\]
Then $du=\zeta\,dt$. As $t$ runs through the interval $(1-\alpha s,1]$, the variable $u$ runs through an interval of length $\alpha\zeta s$. Replacing $q$ by $q+1$ translates this interval by exactly $\alpha\zeta s$. Hence these intervals tile the real line. Therefore
\begin{equation}\label{eq:first-unfold-t}
 \int_{1-\alpha s}^{1}
 \sum_{q\in\Z}
 f\left(sx_{\zeta,a}+t\zeta+q\alpha\zeta s,\frac{\zeta}{s}\right)dt
 =\frac{1}{\zeta}\int_\R f\left(u,\frac{\zeta}{s}\right)du.
\end{equation}
The right-hand side is independent of the representative $a$. Summing over all $\ph(\zeta)$ possible values for $a$ and over all heights gives
\begin{equation}\label{eq:first-before-y}
 \int_{\Om_\alpha}\widehat f_+\,dm_\alpha
 =\frac{2}{\alpha}
 \sum_{\zeta\in J}\frac{\ph(\zeta)}{\zeta}
 \int_0^1\int_\R f\left(x,\frac{\zeta}{s}\right)dx\,ds.
\end{equation}
Now set $y=\zeta/s$. Then $s=\zeta/y$ and $ds=-\zeta y^{-2}dy$. As $s$ increases from $0$ to $1$, the variable $y$ decreases from $+\infty$ to $\zeta$. Hence
\[
 \frac1\zeta\int_0^1\int_\R f\left(x,\frac\zeta s\right)dx\,ds
 =\int_\zeta^\infty\int_\R f(x,y)\,\frac{1}{y^2}\,dx\,dy.
\]
Substitution into \eqref{eq:first-before-y} proves \eqref{eq:first-general}. 
If $f$ changes sign, choose $M,Y>0$ such that $\supp(f)\subset[-M,M]\times(0,Y]$ and apply the nonnegative formula to $|f|$. Only the finitely many heights $\zeta\leq Y$ contribute, and each integral on the right of \eqref{eq:first-general} is bounded by $2M\|f\|_\infty/\zeta$. Thus the transform of $|f|$ is integrable. We may therefore apply the formula separately to the positive and negative parts of $f$. This completes the proof.
\end{proof}

\begin{remark}\label{rem:minus-id}
The assumption $-\Id\in\Gam$ gives the one-triangle parameterization used above, but the unfolding argument also applies when $-\Id\notin\Gam$. Suppose $-\Id\notin\Gam.$ Uyanik--Work distinguish two cases for the geometry of the cross section \cite{UyanikWork2016}.

If a primitive parabolic generator has eigenvalue $1$, then after conjugation it has the form
\[
 \begin{pmatrix}1&\alpha\\0&1\end{pmatrix}.
\]
The two horizontal orientations are no longer identified in $\G/\Gam$, and the section consists of two triangles
\[
 \Om^+=\{(a,b):0<a\leq1,\ 1-\alpha a<b\leq1\}
\]
and
\[
 \Om^-=\{(a,b):-1\leq a<0,\ 1+\alpha a<b\leq1\}.
\]
Their union has Euclidean area $\alpha$, so the induced probability measure has density $1/\alpha$ with respect to $da\,db$. On this two-component section, define the positive transform by summing over vectors whose \emph{transformed} height is positive:
\[
 \widehat f_+(a,b)=\sum_{\substack{v\in\Lam\\\pi_2(p_{a,b}v)>0}}f(p_{a,b}v).
\]
On $\Om^+$ this agrees with the earlier definition. On $\Om^-$ write $a=-s$ and use $-\Lam=\Lam$ to write a contributing vector as $(-u,-\zeta)$, with $(u,\zeta)\in\Lam_+$. Then
\[
 p_{-s,b}(-u,-\zeta)=(su-b\zeta,\zeta/s).
\]
The variable $-b$ runs through an interval of length $\alpha s$, so the same parabolic unfolding applies. The two components have equal integrals, each with coefficient $1/\alpha$, and their sum gives the factor $2/\alpha$ in \eqref{eq:first-general}.

If a primitive parabolic generator has eigenvalue $-1$, replace it by its inverse if necessary and write it as $P=-h_\beta$ with $\beta>0$. Then $P^2=h_{2\beta}$, so the cusp width as defined in this paper is $\alpha=2\beta$. The section is a single triangle
\[
 \Om=\{(a,b):0<a\leq1,\ 1-2\beta a<b\leq1\}
      =\Om_\alpha.
\]
Its probability density is $1/\beta=2/\alpha$. Using representatives modulo $\alpha\zeta$ therefore gives exactly \eqref{eq:first-general}. This distinction between the parameter $\beta$ in $P$ and the positive unipotent width $\alpha$ is needed when comparing the parabolic normal forms in \cite{UyanikWork2016}.

\end{remark}

\begin{remark}\label{rem:several-cusps}
Suppose that $\G/\Gam$ has finitely many cusps. Uyanik--Work decompose the Poincar\'e section as a finite disjoint union
\[
 \Om=\bigsqcup_{i=1}^r\Om_i,
\]
where each $\Om_i$ is associated to one conjugacy class of maximal parabolic subgroups. For the $i$th cusp, choose $C_i\in\G$ sending a shortest vector in the corresponding parabolic direction to $e_1$. The corresponding surfaces are represented by $p_{s,t}C_i\omega$. In these coordinates there is a cusp width $\alpha_i>0$ and a closed, discrete set $C_i\Lam$ with height set $J_i$ and horizontal multiplicity function $\ph_i$. On this component, the positive transform is computed using the vectors of $C_i\Lam$ whose images under $p_{s,t}$ have positive height. The proof of \cref{thm:first-general}, with the orientation convention in \cref{rem:minus-id} when needed, applies to each component separately.

Let $m$ be the invariant probability measure on the full section and put $w_i=m(\Om_i)$. If $m_i=w_i^{-1}m|_{\Om_i}$ denotes the normalized probability measure on the $i$th component, then
\[
 \int_{\Om}\widehat f_+\,dm
 =\sum_{i=1}^r w_i\int_{\Om_i}\widehat f_+\,dm_i.
\]
After conjugating the $i$th cusp to the horizontal direction, its normalized contribution has the same form as \cref{thm:first-general}, with $\alpha$, $J$, and $\ph$ replaced by the cusp data of each cusp.
\end{remark}

From now on, we assume that $\mathrm{SL}(2,\R)/\Gamma$ has one cusp without stating it explicitly.

\subsection{The modular case}

If $\omega$ is the unit area square torus, then $\Gam=\mathrm{SL}(2,\Z)$ and $\Lam=\Z^2_{\prim}$. The positive vectors are
\[
 \Z^2_{\prim,+}
 =\{(A,j)\in\Z^2:j\geq1,\ \gcd(A,j)=1\}.
\]

\begin{corollary}\label{cor:modular-first}
Let $\omega$ be the unit area square torus. Then, for every bounded compactly supported Borel function $f:\R\times(0,\infty)\to\R$,
\begin{equation}\label{eq:modular-first}
 \int_\Om\widehat f_+\,dm
 =2\sum_{j\geq1}\phi(j)
 \int_j^\infty\int_\R f(x,y)\,\frac{1}{y^2}\,dx\,dy.
\end{equation}
Equivalently,
\begin{equation}\label{eq:modular-density}
 \int_\Om\widehat f_+\,dm
 =2\int_0^\infty\int_\R
 f(x,y)\frac{1}{y^2}
 \left(\sum_{1\leq j\leq y}\phi(j)\right)dx\,dy.
\end{equation}
\end{corollary}

\begin{proof}
For the unit area square torus, $\Gam=\mathrm{SL}(2,\Z)$, the cusp width is $\alpha=1$, $J=\N$, and the geometric Euler totient function is the classical Euler totient function. Substituting this into \cref{thm:first-general}, we obtain the desired result.
\end{proof}

\section{The second moment and the second factorial moment}\label{sec:second}

We now compute the second moment. 

\subsection{The second moment for lattice  surfaces}

For a bounded compactly supported Borel function $F$ on $(\R\times(0,\infty))^2$, recall that
\[
 \widehat F^{(2)}(s,t)
 =\sum_{v,w\in\Lam_+}F(p_{s,t}v,p_{s,t}w).
\]
If $F(v,w)=f_1(v)f_2(w)$, then
\[
 \widehat F^{(2)}(s,t)
 =\widehat{f_1}_+(s,t)\widehat{f_2}_+(s,t).
\]
In particular, taking $f_1=f_2=f$ gives $\widehat F^{(2)}=\widehat f_+^{\,2}$. We use a function of two vectors so that the formula also applies when the two test functions differ.

Fix heights $\zeta,\eta\in J$ and representatives indexed by
\[
 1\leq a\leq\ph(\zeta),
 \qquad 1\leq b\leq\ph(\eta).
\]
For $m\in\Z$, define
\begin{equation}\label{eq:D-general-pair}
 D_{\zeta,\eta}(a,b;m)
 =x_{\zeta,a}\eta-x_{\eta,b}\zeta+\alpha\zeta\eta m.
\end{equation}
The first two terms give the determinant of the chosen representatives. The last term records the change in determinant when one vector is shifted relative to the other.

Once the first transformed vector is $(x,y)$ and the determinant is $D$, the second transformed vector is determined by the ratio of the heights. We write
\begin{equation}\label{eq:V-pair-general}
 \begin{aligned}
 V_1(x,y)&=(x,y),\\
 V_2^{\zeta,\eta,D}(x,y)
 &=\left(\frac{\eta}{\zeta}x-\frac{D}{y},
          \frac{\eta}{\zeta}y\right).
 \end{aligned}
\end{equation}
Notice that $\detm(V_1,V_2^{\zeta,\eta,D})=D$.

\begin{theorem}\label{thm:second-general}
Let $\omega$ be a lattice surface whose Veech group $\Gam$ has one cusp and contains $-\Id$. Use the cusp data and coordinates of \cref{sec:first}. Then, for every bounded compactly supported Borel function $F:(\R\times(0,\infty))^2\to\R$,
\begin{align}\label{eq:second-general}
 \int_{\Om_\alpha}\widehat F^{(2)}\,dm_\alpha
 &=\frac{2}{\alpha}
 \sum_{\zeta,\eta\in J}
 \sum_{a=1}^{\ph(\zeta)}
 \sum_{b=1}^{\ph(\eta)}
 \sum_{m\in\Z}\notag\\
 &\qquad\times\int_\zeta^\infty\int_\R
 F\left(V_1(x,y),V_2^{\zeta,\eta,D}(x,y)\right)
 \frac{1}{y^2}\,dx\,dy,
\end{align}
where $D=D_{\zeta,\eta}(a,b;m)$.
\end{theorem}

\begin{proof}
First suppose that $F$ is nonnegative. We will decompose the integral into contributions from pairs with fixed heights, representatives, and relative parabolic shift, and then compute each contribution. All exchanges of sums and integrals below are justified by Tonelli's theorem.

By \eqref{eq:height-decomp}, every ordered pair in $\Lam_+^2$ has a unique representation
\[
 v=(x_{\zeta,a}+q_1\alpha\zeta,\zeta),
 \qquad
 w=(x_{\eta,b}+q_2\alpha\eta,\eta),
\]
where $\zeta,\eta\in J$, $1\leq a\leq\ph(\zeta)$, $1\leq b\leq\ph(\eta)$, and $q_1,q_2\in\Z$. Set $m=q_1-q_2$ and $q=q_1$. The change of indices $(q_1,q_2)\leftrightarrow(q,m)$ is a bijection of $\Z^2$. For fixed $\zeta,\eta,a,b,m$, the resulting family of pairs is
\[
 v_q=(x_{\zeta,a}+q\alpha\zeta,\zeta),
 \qquad
 w_{q-m}=(x_{\eta,b}+(q-m)\alpha\eta,\eta),
 \qquad q\in\Z.
\]
Denote the contribution of this family by
\begin{equation}\label{eq:second-family-contribution}
 C_{\zeta,\eta}(a,b;m)
 =\frac{2}{\alpha}\int_0^1\int_{1-\alpha s}^{1}
 \sum_{q\in\Z}F(p_{s,t}v_q,p_{s,t}w_{q-m})\,dt\,ds.
\end{equation}
Using the section measure $dm_\alpha=\frac{2}{\alpha}\,ds\,dt$, we can therefore write the full integral as
\begin{equation}\label{eq:second-decomposition}
 \int_{\Om_\alpha}\widehat F^{(2)}\,dm_\alpha
 =\sum_{\zeta,\eta\in J}
 \sum_{a=1}^{\ph(\zeta)}\sum_{b=1}^{\ph(\eta)}
 \sum_{m\in\Z}C_{\zeta,\eta}(a,b;m).
\end{equation}

We now compute one such contribution. Fix $\zeta,\eta,a,b,m$ and write $D=D_{\zeta,\eta}(a,b;m)$. Every pair in this family has the same determinant, since
\begin{equation}\label{eq:det-pair-proof}
 \begin{aligned}
 \detm(v_q,w_{q-m})
 &=(x_{\zeta,a}+q\alpha\zeta)\eta
   -(x_{\eta,b}+(q-m)\alpha\eta)\zeta\\
 &=x_{\zeta,a}\eta-x_{\eta,b}\zeta+\alpha\zeta\eta m
 =D.
 \end{aligned}
\end{equation}
Write the first transformed vector as $p_{s,t}v_q=(x,y)$, where
\begin{equation}\label{eq:xy-pair}
 x=s(x_{\zeta,a}+q\alpha\zeta)+t\zeta,
 \qquad y=\frac{\zeta}{s}.
\end{equation}
The first coordinate of the second transformed vector satisfies
\begin{align*}
 s(x_{\eta,b}+(q-m)\alpha\eta)+t\eta
 &=\frac{\eta}{\zeta}
   \bigl(s(x_{\zeta,a}+q\alpha\zeta)+t\zeta\bigr)
   -\frac{sD}{\zeta}\\
 &=\frac{\eta}{\zeta}x-\frac{D}{y}.
\end{align*}
Its second coordinate is $\eta/s=(\eta/\zeta)y$. Thus
\[
 p_{s,t}v_q=V_1(x,y),
 \qquad
 p_{s,t}w_{q-m}=V_2^{\zeta,\eta,D}(x,y).
\]

We first integrate in $t$, keeping $s$ fixed. For each $q\in\Z$, the substitution in \eqref{eq:xy-pair} gives $dx=\zeta\,dt$ and sends $(1-\alpha s,1]$ to
\[
 I_q(s)
 =\bigl(sx_{\zeta,a}+\zeta+(q-1)\alpha\zeta s,
        sx_{\zeta,a}+\zeta+q\alpha\zeta s\bigr].
\]
These intervals have length $\alpha\zeta s$ and tile $\R$ as $q$ varies. Since $D$ is independent of $q$, we obtain
\begin{equation}\label{eq:second-unfold-t}
 \begin{aligned}
 &\int_{1-\alpha s}^{1}
   \sum_{q\in\Z}F(p_{s,t}v_q,p_{s,t}w_{q-m})\,dt\\
 &\qquad=\frac{1}{\zeta}\sum_{q\in\Z}\int_{I_q(s)}
   F\left(V_1\left(x,\frac{\zeta}{s}\right),
          V_2^{\zeta,\eta,D}\left(x,\frac{\zeta}{s}\right)\right)dx\\
 &\qquad=\frac{1}{\zeta}\int_\R
   F\left(V_1\left(x,\frac{\zeta}{s}\right),
          V_2^{\zeta,\eta,D}\left(x,\frac{\zeta}{s}\right)\right)dx.
 \end{aligned}
\end{equation}
Substituting this into \eqref{eq:second-family-contribution} gives
\begin{equation}\label{eq:second-family-before-y}
 C_{\zeta,\eta}(a,b;m)
 =\frac{2}{\alpha\zeta}\int_0^1\int_\R
 F\left(V_1\left(x,\frac{\zeta}{s}\right),
        V_2^{\zeta,\eta,D}\left(x,\frac{\zeta}{s}\right)\right)dx\,ds.
\end{equation}

It remains to integrate in $s$. As in the proof of \cref{thm:first-general}, set $y=\zeta/s$. Then $ds=-\zeta y^{-2}\,dy$, and $y$ decreases from $+\infty$ to $\zeta$ as $s$ increases from $0$ to $1$. Consequently, each family contributes exactly
\begin{equation}\label{eq:second-family-evaluated}
 C_{\zeta,\eta}(a,b;m)
 =\frac{2}{\alpha}\int_\zeta^\infty\int_\R
 F\left(V_1(x,y),V_2^{\zeta,\eta,D}(x,y)\right)
 \frac{1}{y^2}\,dx\,dy.
\end{equation}
Summing \eqref{eq:second-family-evaluated} over the parameters in \eqref{eq:second-decomposition} proves \eqref{eq:second-general} for nonnegative $F$.

Finally, we check finiteness so that the formula also applies to signed functions. Choose $M,Y>0$ so that every vector in either projection of $\supp(F)$ to $\R^2$ lies in $[-M,M]\times(0,Y]$. A contributing pair has $\zeta,\eta\leq Y$, so only finitely many heights and representatives occur. Moreover,
\[
 |D|=|\detm(V_1,V_2^{\zeta,\eta,D})|\leq2MY.
\]
For each fixed choice of heights and representatives, \eqref{eq:D-general-pair} therefore allows only finitely many integers $m$. Each remaining integral satisfies
\[
 \int_\zeta^\infty\int_\R
 \left|F\left(V_1(x,y),V_2^{\zeta,\eta,D}(x,y)\right)\right|
 \frac{1}{y^2}\,dx\,dy
 \leq\frac{2M\|F\|_\infty}{\zeta}<\infty.
\]
Applying the nonnegative formula to $|F|$ proves absolute integrability. We may therefore apply the formula separately to the positive and negative parts of $F$, completing the proof.
\end{proof}

To compute the second factorial moment, we must remove pairs in which the same vector occurs twice. For visible vectors of positive height, these are exactly the pairs with determinant zero.

\begin{lemma}\label{lem:pair-diagonal-general}
Let $v,w\in\Lam_+$ be visible holonomy vectors with positive heights. Then
\[
 v=w\quad \text{ if and only if }\quad \detm(v,w)=0.
\]
\end{lemma}

\begin{proof}
If $v=w$, then $\detm(v,w)=0$. Conversely, if the determinant is zero, then $v$ and $w$ lie on the same ray because both have positive second coordinate. Thus $w=cv$ for some $c>0$. If $c\neq1$, one of the two vectors is a strict positive contraction of the other, contradicting visibility. Hence $c=1$ and $v=w$.
\end{proof}

Define the second factorial transform
\[
 \widehat F^{[2]}(s,t)
 =\sum_{\substack{v,w\in\Lam_+\\v\neq w}}F(p_{s,t}v,p_{s,t}w).
\]

\begin{corollary}\label{cor:second-factorial-general}
Under the hypotheses of \cref{thm:second-general},
\begin{align}\label{eq:second-factorial-general}
 \int_{\Om_\alpha}\widehat F^{[2]}\,dm_\alpha
 &=\frac{2}{\alpha}
 \sum_{\zeta,\eta\in J}
 \sum_{a=1}^{\ph(\zeta)}
 \sum_{b=1}^{\ph(\eta)}
 \sum_{\substack{m\in\Z\\D_{\zeta,\eta}(a,b;m)\neq0}}\notag\\
 &\qquad\times\int_\zeta^\infty\int_\R
 F\left(V_1(x,y),V_2^{\zeta,\eta,D}(x,y)\right)
 \frac{1}{y^2}\,dx\,dy.
\end{align}
\end{corollary}

\begin{proof}
By Lemma~\ref{lem:pair-diagonal-general}, removing the diagonal $v=w$ is exactly the same as removing the terms with determinant $D=0$ from \cref{thm:second-general}.
\end{proof}

\subsection{Second moment for the modular case}\label{subsec:second-modular}

For the square torus, the heights and determinants are integers. We can therefore collect the terms in the preceding formula according to their determinant. The coefficient below counts how many choices of invertible residue classes give the same determinant. For $j,k\in\N$ and $n\in\Z$, define
\begin{equation}\label{eq:Phi-jkn}
 \Phi_{j,k}(n)
 =\#\left\{(a,b)\in(\Z/j\Z)^\times\times(\Z/k\Z)^\times:
 ak-bj\equiv n\pmod{jk}\right\}.
\end{equation}
For $j=1$ or $k=1$, the corresponding unit group is understood to contain its unique residue class.

\begin{theorem}\label{thm:modular-second}
Let $\omega$ be the unit area square torus. Then for every bounded compactly supported Borel function $F:(\R\times(0,\infty))^2\to\R$,
\begin{align}\label{eq:modular-second}
 &\int_\Om \widehat{F}^{(2)}(s,t)\,dm(s,t)\notag\\
 &\qquad=
 2\sum_{j,k\geq1}\sum_{n\in\Z}
 \Phi_{j,k}(n)
 \int_j^\infty\int_\R
 F\left(
 (x,y),
 \left(\frac{k}{j}x-\frac{n}{y},\frac{k}{j}y\right)
 \right)\frac{1}{y^2}\,dx\,dy.
\end{align}

\end{theorem}

\begin{proof}
Suppose that $F$ is nonnegative. Every vector in $\Z^2_{\prim,+}$ has the form $(A,j)$, where $j\geq1$ and $\gcd(A,j)=1$. Expanding the transform and using $dm=2\,ds\,dt$, we obtain
\begin{equation}\label{eq:modular-expand-pairs}
 \begin{aligned}
 \int_\Om\widehat F^{(2)}\,dm
 &=2\sum_{j,k\geq1}
 \sum_{\substack{A,B\in\Z\\\gcd(A,j)=\gcd(B,k)=1}}
 \int_0^1\int_{1-s}^{1}\\
 &\qquad\qquad
 F\left(\left(sA+tj,\frac{j}{s}\right),
         \left(sB+tk,\frac{k}{s}\right)\right)dt\,ds.
 \end{aligned}
\end{equation}
Here $j$ and $k$ are the heights of the two integer vectors, and $A$ and $B$ are their first coordinates. Tonelli's theorem justifies the exchanges of sums and integrals in this proof.

Fix $j,k\geq1$. Write
\[
 A=a+rj,
 \qquad B=b+qk,
\]
where $a,b$ are the unique integer remainders satisfying
\[
 0\leq a<j,\qquad 0\leq b<k,
 \qquad\gcd(a,j)=\gcd(b,k)=1,
\]
and $r,q\in\Z$. When $j=1$, we take $a=0$, and when $k=1$, we take $b=0$. The determinant of the two vectors is
\begin{equation}\label{eq:modular-determinant-remainders}
 n=\detm\bigl((A,j),(B,k)\bigr)
  =Ak-Bj=ak-bj+jk(r-q).
\end{equation}
Thus the remainders of a pair with determinant $n$ satisfy
\[
 ak-bj\equiv n\pmod{jk}.
\]
Conversely, fix $n\in\Z$ and remainders $a,b$ satisfying this congruence and the coprimality conditions above. Then
\[
 d=\frac{n-(ak-bj)}{jk}
\]
is an integer, and \eqref{eq:modular-determinant-remainders} holds exactly when $r-q=d$. Taking $\ell=q$, all pairs with these heights, remainders, and determinant are therefore
\begin{equation}\label{eq:modular-pairs-fixed-remainders}
 (A_\ell,j),\ (B_\ell,k),
 \qquad
 A_\ell=a+(\ell+d)j,\quad B_\ell=b+\ell k,
 \qquad\ell\in\Z.
\end{equation}
Each pair occurs for exactly one $\ell$. By \eqref{eq:Phi-jkn}, the number of choices of $a,b$ for fixed $j,k,n$ is $\Phi_{j,k}(n)$.

We now compute the contribution to \eqref{eq:modular-expand-pairs} from one fixed choice of $j,k,n,a,b$. We first integrate in $t$, keeping $s$ fixed. For each $\ell\in\Z$, set
\[
 x=sA_\ell+tj.
\]
Then $dx=j\,dt$, and the first transformed vector is $(x,j/s)$. Since $A_\ell k-B_\ell j=n$, the second transformed vector has first coordinate
\begin{align*}
 sB_\ell+tk
 &=\frac{k}{j}(sA_\ell+tj)
   +s\left(B_\ell-\frac{k}{j}A_\ell\right)\\
 &=\frac{k}{j}x-\frac{sn}{j}.
\end{align*}
Its second coordinate is $k/s$. Thus, in these coordinates, the value of $F$ is
\[
 F\left(\left(x,\frac{j}{s}\right),
         \left(\frac{k}{j}x-\frac{sn}{j},\frac{k}{s}\right)\right),
\]
which depends on $j,k,n,s,x$, but not on $a,b$, or $\ell$.

As $t$ runs through $(1-s,1]$, the variable $x$ runs through
\[
 I_\ell(s)=\bigl(sA_\ell+j(1-s),\ sA_\ell+j\bigr].
\]
This interval has length $sj$. Since $A_{\ell+1}=A_\ell+j$, the left endpoint of $I_{\ell+1}(s)$ equals the right endpoint of $I_\ell(s)$. These intervals therefore partition $\R$ as $\ell$ ranges over $\Z$. It follows that
\begin{equation}\label{eq:modular-unfold-t}
 \begin{aligned}
 &\int_{1-s}^{1}\sum_{\ell\in\Z}
 F\bigl(p_{s,t}(A_\ell,j),p_{s,t}(B_\ell,k)\bigr)\,dt\\
 &\quad=\frac{1}{j}\sum_{\ell\in\Z}\int_{I_\ell(s)}
 F\left(\left(x,\frac{j}{s}\right),
         \left(\frac{k}{j}x-\frac{sn}{j},\frac{k}{s}\right)\right)dx\\
 &\quad=\frac{1}{j}\int_\R
 F\left(\left(x,\frac{j}{s}\right),
         \left(\frac{k}{j}x-\frac{sn}{j},\frac{k}{s}\right)\right)dx.
 \end{aligned}
\end{equation}
After integrating in $s$ and including the factor $2$ from $dm$, the contribution of the pairs in \eqref{eq:modular-pairs-fixed-remainders} is therefore
\begin{equation}\label{eq:modular-one-remainder-contribution}
 \frac{2}{j}\int_0^1\int_\R
 F\left(\left(x,\frac{j}{s}\right),
         \left(\frac{k}{j}x-\frac{sn}{j},\frac{k}{s}\right)\right)dx\,ds.
\end{equation}
This contribution is the same for each of the $\Phi_{j,k}(n)$ choices of $a,b$. Summing over these choices, then over the determinants $n$ and the heights $j,k$, gives
\begin{equation}\label{eq:modular-second-before-y}
 \begin{aligned}
 \int_\Om\widehat F^{(2)}\,dm
 &=2\sum_{j,k\geq1}\sum_{n\in\Z}\frac{\Phi_{j,k}(n)}{j}
 \int_0^1\int_\R\\
 &\qquad\qquad
 F\left(\left(x,\frac{j}{s}\right),
         \left(\frac{k}{j}x-\frac{sn}{j},\frac{k}{s}\right)\right)dx\,ds.
 \end{aligned}
\end{equation}

Now set $y=j/s$, as in the proof of \cref{thm:first-general}. Then $ds=-jy^{-2}\,dy$, and $y$ decreases from $+\infty$ to $j$ as $s$ increases from $0$ to $1$. Also, $sn/j=n/y$ and $k/s=(k/j)y$. Hence
\begin{equation}\label{eq:modular-one-remainder-evaluated}
 \begin{aligned}
 &\frac{2}{j}\int_0^1\int_\R
 F\left(\left(x,\frac{j}{s}\right),
         \left(\frac{k}{j}x-\frac{sn}{j},\frac{k}{s}\right)\right)dx\,ds\\
 &\qquad=2\int_j^\infty\int_\R
 F\left((x,y),
         \left(\frac{k}{j}x-\frac{n}{y},\frac{k}{j}y\right)\right)
 \frac{1}{y^2}\,dx\,dy.
 \end{aligned}
\end{equation}
Substituting this identity into \eqref{eq:modular-second-before-y} proves \eqref{eq:modular-second} for nonnegative $F$.

If $F$ changes sign, choose $M,Y>0$ so that both vectors in every pair belonging to $\supp(F)$ lie in $[-M,M]\times(0,Y]$. A nonzero integrand on the right-hand side of \eqref{eq:modular-second} then requires $j\leq y\leq Y$ and $k\leq(k/j)y\leq Y$. It also requires $|n|\leq2MY$, since $n$ is the determinant of the two vectors. Thus only finitely many triples $(j,k,n)$ contribute. For each such triple, the integral with $|F|$ in place of $F$ is at most
\[
 2M\|F\|_\infty\int_j^\infty\frac{1}{y^2}\,dy
 =\frac{2M\|F\|_\infty}{j}<\infty.
\]
The formula for $|F|$ therefore proves absolute integrability. Applying the formula separately to the positive and negative parts of $F$ completes the proof.
\end{proof}

\begin{corollary}\label{thm:modular-second-factorial}
With the notation of \cref{thm:modular-second}, every bounded compactly supported Borel function $F:(\R\times(0,\infty))^2\to\R$ satisfies

\begin{align}\label{eq:modular-second-factorial}
 &\int_\Om\sum_{\substack{v,w\in\Z^2_{\prim,+}\\v\neq w}}
 F(p_{s,t}v,p_{s,t}w)\,dm(s,t)\notag\\
 &\qquad=
 2\sum_{j,k\geq1}\sum_{n\in\Z\setminus\{0\}}
 \Phi_{j,k}(n)
 \int_j^\infty\int_\R
 F\left(
 (x,y),
 \left(\frac{k}{j}x-\frac{n}{y},\frac{k}{j}y\right)
 \right)\frac{1}{y^2}\,dx\,dy.
\end{align}
\end{corollary}

\begin{proof}
Primitive integer vectors are visible, so \cref{lem:pair-diagonal-general} applies. Thus two positive primitive vectors are distinct exactly when their determinant $n$ is nonzero. Removing the terms with $n=0$ from \eqref{eq:modular-second} gives the formula.

\end{proof}


\section{Higher moments and higher factorial moments}\label{sec:higher}

The second-moment proof extends to any number of vectors. We use the first vector to choose the integration variables. Each remaining vector is then determined by its height and its determinant with the first vector. As before, one common parabolic shift accounts for the integral over the real line.

\subsection{Higher moments}

We continue with the one-cusp assumptions and the notation of \cref{sec:first}. Fix $r\in\N$. For heights $\zeta_1,\ldots,\zeta_r\in J$, not necessarily distinct, representative indices
\[
 1\leq a_\ell\leq\ph(\zeta_\ell),
 \qquad 1\leq\ell\leq r,
\]
and integers $m_2,\ldots,m_r\in\Z$, define
\begin{equation}\label{eq:D-ell-general}
 D_\ell
 =x_{\zeta_1,a_1}\zeta_\ell
 -x_{\zeta_\ell,a_\ell}\zeta_1
 +\alpha\zeta_1\zeta_\ell m_\ell,
 \qquad 2\leq\ell\leq r.
\end{equation}
Set
\begin{equation}\label{eq:V-r-general}
 V_1(x,y)=(x,y),
 \qquad
 V_\ell(x,y)
 =\left(
 \frac{\zeta_\ell}{\zeta_1}x-\frac{D_\ell}{y},
 \frac{\zeta_\ell}{\zeta_1}y
 \right),
 \quad 2\leq\ell\leq r.
\end{equation}
When $r=1$, the determinant and relative-shift data are empty, and a sum over the empty tuple has one term.

For a bounded compactly supported Borel function $F$ on $(\R\times(0,\infty))^r$, define
\[
 \widehat F^{(r)}(s,t)
 =\sum_{v_1,\ldots,v_r\in\Lam_+}
 F(p_{s,t}v_1,\ldots,p_{s,t}v_r).
\]

\Needspace{12\baselineskip}
\begin{theorem}\label{thm:r-general}
Let $\omega$ be a one-cusp lattice surface with $-\Id\in\Gam$. For every $r\in\N$ and every bounded compactly supported Borel function $F:(\R\times(0,\infty))^r\to\R$,
\begin{align}\label{eq:r-general}
 \int_{\Om_\alpha}\widehat F^{(r)}\,dm_\alpha
 &=\frac{2}{\alpha}
 \sum_{\zeta_1,\ldots,\zeta_r\in J}
 \sum_{a_1=1}^{\ph(\zeta_1)}\cdots
 \sum_{a_r=1}^{\ph(\zeta_r)}
 \sum_{m_2,\ldots,m_r\in\Z}\notag\\
 &\qquad\times
 \int_{\zeta_1}^\infty\int_\R
 F\bigl(V_1(x,y),\ldots,V_r(x,y)\bigr)
 \frac{1}{y^2}\,dx\,dy.
\end{align}

\end{theorem}

\begin{proof}
First suppose that $F\geq0$, so Tonelli's theorem applies. Write
\[
 v_\ell=(x_{\zeta_\ell,a_\ell}+q_\ell\alpha\zeta_\ell,\zeta_\ell).
\]
Use $v_1$ as the first vector and define
\[
 m_\ell=q_1-q_\ell,
 \qquad 2\leq\ell\leq r.
\]
Then
\[
 \detm(v_1,v_\ell)
 =x_{\zeta_1,a_1}\zeta_\ell
 -x_{\zeta_\ell,a_\ell}\zeta_1
 +\alpha\zeta_1\zeta_\ell(q_1-q_\ell)
 =D_\ell.
\]
As in the proof of \cref{thm:second-general}, once the heights, parabolic orbit representatives, and relative shifts $m_2,\ldots,m_r$ are fixed, exactly one free integer parameter remains, namely the common shift $q_1$.

Set
\[
 x=s(x_{\zeta_1,a_1}+q_1\alpha\zeta_1)+t\zeta_1,
 \qquad
 y=\frac{\zeta_1}{s}.
\]
The $t$-interval has length $\alpha s$, so for fixed $q_1$ the variable $x$ runs over an interval of length $\alpha\zeta_1s$. Increasing $q_1$ by one translates that interval by the same quantity. Thus the common shift unfolds the horizontal coordinate to all of $\R$.

For $\ell\geq2$, the determinant identity gives
\[
 x_{\zeta_\ell,a_\ell}+q_\ell\alpha\zeta_\ell
 =\frac{\zeta_\ell}{\zeta_1}
 (x_{\zeta_1,a_1}+q_1\alpha\zeta_1)
 -\frac{D_\ell}{\zeta_1}.
\]
Applying $p_{s,t}$ and using $s/\zeta_1=1/y$ gives \eqref{eq:V-r-general}. The Jacobian is the same as in the two-vector case,
\[
 ds\,dt=\frac{1}{y^2}\,dx\,dy,
\]
and $s\leq1$ becomes $y\geq\zeta_1$. Multiplication by $2/\alpha$ gives \eqref{eq:r-general}.

For a signed function $F$, choose $M,Y>0$ so that every vector in every tuple in $\supp(F)$ lies in $[-M,M]\times(0,Y]$. Any contributing term has $\zeta_\ell\leq Y$ for every $\ell$ and $|D_\ell|\leq2MY$ for $\ell\geq2$. There are only finitely many such heights and representatives, and \eqref{eq:D-ell-general} then permits only finitely many relative shifts. Each remaining integral with $|F|$ is at most $2M\|F\|_\infty/\zeta_1$. Applying the nonnegative formula to $|F|$ proves absolute integrability, so the result follows by applying the formula to the positive and negative parts of $F$.
\end{proof}

For factorial moments, we must determine when the vectors in a tuple are pairwise distinct. Requiring $D_\ell\neq0$ ensures that $v_\ell$ differs from $v_1$, but we must also compare the remaining vectors with one another. The following lemma gives all of these conditions in terms of the same determinants.

\begin{lemma}\label{lem:distinct-general}
Let $v_1,\ldots,v_r\in\Lam_+$ be visible vectors and put
\[
 \zeta_\ell=\pi_2(v_\ell)\quad(1\leq\ell\leq r),
 \qquad D_\ell=\detm(v_1,v_\ell)\quad(2\leq\ell\leq r).
\]
Then $v_1,\ldots,v_r$ are pairwise distinct if and only if
\begin{equation}\label{eq:distinct-general-a}
 D_\ell\neq0,
 \qquad 2\leq\ell\leq r,
\end{equation}
and
\begin{equation}\label{eq:distinct-general-b}
 \zeta_\ell D_m-\zeta_mD_\ell\neq0,
 \qquad 2\leq\ell<m\leq r.
\end{equation}
\end{lemma}

\begin{proof}
By Lemma~\ref{lem:pair-diagonal-general}, $v_1=v_\ell$ exactly when $D_\ell=0$. For $2\leq\ell<m\leq r$, write $v_i=(A_i,\zeta_i)$. Since
\[
 D_\ell=A_1\zeta_\ell-A_\ell\zeta_1,
 \qquad
 D_m=A_1\zeta_m-A_m\zeta_1,
\]
the definitions give
\begin{equation}\label{eq:det-lm-general}
 \detm(v_\ell,v_m)
 =A_\ell\zeta_m-A_m\zeta_\ell
 =\frac{\zeta_\ell D_m-\zeta_mD_\ell}{\zeta_1}.
\end{equation}
Visibility again implies that $v_\ell=v_m$ if and only if this determinant is zero. This proves the criterion.
\end{proof}

Let $\cA(\zeta_1,\ldots,\zeta_r;a_1,\ldots,a_r)$ be the set of integer tuples $(m_2,\ldots,m_r)$ for which \eqref{eq:distinct-general-a} and \eqref{eq:distinct-general-b} hold.

\begin{theorem}\label{thm:factorial-r-general}
Under the hypotheses of \cref{thm:r-general}, let $r\in\N$ with $r\geq2$, and let $F:(\R\times(0,\infty))^r\to\R$ be a bounded compactly supported Borel function. Then
\begin{align}\label{eq:factorial-r-general}
 &\int_{\Om_\alpha}
 \sum_{v_1,\ldots,v_r\in\Lam_+}^{\neq}
 F(p_{s,t}v_1,\ldots,p_{s,t}v_r)\,dm_\alpha(s,t)\notag\\
 &=\frac{2}{\alpha}
 \sum_{\zeta_1,\ldots,\zeta_r\in J}
 \sum_{a_1=1}^{\ph(\zeta_1)}\cdots
 \sum_{a_r=1}^{\ph(\zeta_r)}
 \sum_{(m_2,\ldots,m_r)\in\cA}\notag\\
 &\qquad\times\int_{\zeta_1}^\infty\int_\R
 F\bigl(V_1(x,y),\ldots,V_r(x,y)\bigr)
 \frac{1}{y^2}\,dx\,dy.
\end{align}
\end{theorem}

\begin{proof}
The tuple unfolding in \cref{thm:r-general} partitions all ordered vector tuples according to their heights, parabolic orbit representatives, and relative shifts. By Lemma~\ref{lem:distinct-general}, the tuple is pairwise distinct exactly when the relative shifts lie in $\cA$. We therefore keep exactly these relative shifts in the sum from \cref{thm:r-general}. The change of variables and its Jacobian remain the same.
\end{proof}

\subsection{Higher factorial moments for the modular case}\label{subsec:higher-modular}

For the square torus, we can again combine all residue choices that give the same determinants. This is the higher-moment version of the coefficient $\Phi_{j,k}(n)$ in \cref{subsec:second-modular}. For positive integers $j_1,\ldots,j_r$ and integers $n_2,\ldots,n_r$, define
\begin{align}\label{eq:Phi-general-modular}
 \Phi_{j_1,\ldots,j_r}(n_2,\ldots,n_r)
 =\#\Bigl\{&(a_1,\ldots,a_r)
 \in\prod_{\ell=1}^r(\Z/j_\ell\Z)^\times:\notag\\
 &a_1j_\ell-a_\ell j_1\equiv n_\ell
 \pmod{j_1j_\ell},\quad 2\leq\ell\leq r
 \Bigr\}.
\end{align}
Define
\begin{equation}\label{eq:V-general-modular}
 V_1(x,y)=(x,y),
 \qquad
 V_\ell(x,y)
 =\left(\frac{j_\ell}{j_1}x-\frac{n_\ell}{y},
 \frac{j_\ell}{j_1}y\right),
 \quad 2\leq\ell\leq r.
\end{equation}
The congruences in \eqref{eq:Phi-general-modular} decide whether these determinant data can arise from primitive vectors. We impose distinctness separately. Let $\cA(j_1,\ldots,j_r)$ be the set of determinant tuples $(n_2,\ldots,n_r)\in\Z^{r-1}$ such that
\begin{equation}\label{eq:mod-distinct-1}
 n_\ell\neq0,
 \qquad 2\leq\ell\leq r,
\end{equation}
and
\begin{equation}\label{eq:mod-distinct-2}
 j_\ell n_m-j_mn_\ell\neq0,
 \qquad 2\leq\ell<m\leq r.
\end{equation}

\Needspace{12\baselineskip}
\begin{theorem}\label{thm:modular-higher-factorial}
Let $\Om$ be the BCZ triangle with $dm=2\,ds\,dt$, and let $\Lam=\Z^2_{\prim}$. Let $r\in\N$ with $r\geq2$, and let $F:(\R\times(0,\infty))^r\to\R$ be a bounded compactly supported Borel function. Then
\begin{align}\label{eq:modular-higher-factorial}
 &\int_\Om
 \sum_{v_1,\ldots,v_r\in\Z^2_{\prim,+}}^{\neq}
 F(p_{s,t}v_1,\ldots,p_{s,t}v_r)\,dm(s,t)\notag\\
 &=2\sum_{j_1,\ldots,j_r\geq1}
 \sum_{(n_2,\ldots,n_r)\in\cA(j_1,\ldots,j_r)}
 \Phi_{j_1,\ldots,j_r}(n_2,\ldots,n_r)\notag\\
 &\qquad\times
 \int_{j_1}^\infty\int_\R
 F\bigl(V_1(x,y),\ldots,V_r(x,y)\bigr)
 \frac{1}{y^2}\,dx\,dy.
\end{align}
\end{theorem}

\begin{proof}
At height $j_\ell$, write
\[
 v_\ell=(a_\ell+q_\ell j_\ell,j_\ell),
 \qquad
 0\leq a_\ell<j_\ell,\qquad \gcd(a_\ell,j_\ell)=1.
\]
When $j_\ell=1$, we take $a_\ell=0$.
For $\ell\geq2$, define the determinant with the first vector by
\begin{equation}\label{eq:n-ell-modular}
 n_\ell=\detm(v_1,v_\ell)
 =a_1j_\ell-a_\ell j_1+j_1j_\ell(q_1-q_\ell).
\end{equation}
A residue tuple $(a_1,\ldots,a_r)$ contributes to fixed determinant data $(n_2,\ldots,n_r)$ exactly when the congruences in \eqref{eq:Phi-general-modular} hold. If they hold, each difference $q_1-q_\ell$ is uniquely determined by \eqref{eq:n-ell-modular}. Hence one common shift $q_1$ remains. Unfolding that common shift produces $x\in\R$ exactly as in \cref{thm:r-general}.

The determinant identity gives the transformed vectors in the form \eqref{eq:V-general-modular}. Finally, the pairwise distinctness criterion \eqref{eq:mod-distinct-1}--\eqref{eq:mod-distinct-2} is the specialization of Lemma~\ref{lem:distinct-general}. Grouping the one-cusp formula by determinant data proves the theorem.
\end{proof}


\section{Farey pair correlation from the first cross-sectional moment}\label{sec:farey}

Boca and Zaharescu proved the existence of all correlation measures of Farey fractions and computed the pair-correlation density explicitly \cite{BocaZaharescu2005}. We recover their formula using \cref{thm:first-general} and dynamical methods.

\subsection{Return time statistics}\label{subsec:bcz}

For $Q\in\N$, the Farey sequence of order $Q$ consists of the reduced fractions in $[0,1]$ with denominator at most $Q$:
\begin{equation}\label{eq:farey-sequence}
\begin{aligned}
\cF_Q
&=\left\{\frac{a}{q}: a,q\in\Z,\ 1\leq q\leq Q,\ 
0\leq a\leq q,\ \gcd(a,q)=1\right\} \\
&=\{0=\gamma_0<\gamma_1<\cdots<\gamma_{N_Q}=1\}.
\end{aligned}
\end{equation}

The BCZ map on $\Omega$ is defined by

\begin{equation}\label{eq:bcz-map-return-time}
 T(s,t)=\left(t,\left\lfloor\frac{1+s}{t}\right\rfloor t-s\right).
\end{equation}

For $(s,t)\in\Om$, let $p=p_{s,t}$. A primitive vector $v\in\Z^2_{\prim,+}$ with
\[
 pv=(x,y),
 \qquad 0<x\leq1,
\]
becomes horizontal under the horocycle flow at time
\begin{equation}\label{eq:hitting-time}
 \tau(v)=\frac{y}{x}.
\end{equation}
Conversely, every positive return to the section comes from such a vector. The vector is unique because a lattice has only one primitive vector on each positive ray. Thus counting future return times is equivalent to counting these primitive vectors.

\subsection{Periodic BCZ orbits}
Write $N_Q=|\mathcal F_Q|-1$, the number of Farey fractions in $[0,1)$. Then

\begin{equation}\label{eq:NQ-asymp}
 N_Q=\sum_{q=1}^Q\phi(q),
 \qquad \lim_{Q\to\infty}\frac{N_Q}{Q^2}=\frac{3}{\pi^2}.
\end{equation}

Write $\gamma_i=a_i/q_i$ in lowest terms. Extend the sequence by $\gamma_{i+N_Q}=\gamma_i+1$ and the denominators by $q_{i+N_Q}=q_i$. The BCZ map satisfies
\[
T\left(\frac{q_i}{Q},\frac{q_{i+1}}{Q}\right)=\left(\frac{q_{i+1}}{Q},\frac{q_{i+2}}{Q}\right).
\]

Set
\begin{equation}\label{eq:farey-orbit-coordinates}
 z_{Q,i}=\left(\frac{q_i}{Q},\frac{q_{i+1}}{Q}\right)
        =T^i\left(\frac1Q,1\right).
\end{equation}

The point $z_{Q,i}$ has period $N_Q$ under $T$. By \cite[Theorem~1.3]{AthreyaCheung2014}, the probability measures
\[
 \nu_Q=\frac1{N_Q}\sum_{i=0}^{N_Q-1}\delta_{z_{Q,i}}
\]
converge weakly to $m$ as $Q\to\infty$.

Athreya and Cheung \cite{AthreyaCheung2014} showed that the first-return time of the horocycle flow at every point $(s,t)\in\Omega$ is
\[
 R(s,t)=\frac{1}{st}\geq1.
\]
In particular, the first-return time at $z_{Q,i}$ is $R(z_{Q,i})$. Let $\tau_{i,r}$ be the time of the $r$th future return, with $\tau_{i,0}=0$. The identity $\gamma_{i+1}-\gamma_i=1/(q_iq_{i+1})$ for consecutive Farey fractions gives

\begin{equation}\label{eq:rth-future-time}
 \begin{aligned}
 \tau_{i,r}
 &=\sum_{\ell=0}^{r-1}R(z_{Q,i+\ell})
  =Q^2\sum_{\ell=0}^{r-1}\frac{1}{q_{i+\ell}q_{i+\ell+1}}\\
 &=Q^2(\gamma_{i+r}-\gamma_i).
 \end{aligned}
\end{equation}

\subsection{Pair correlation formula}\label{subsec:pair-correlation}
Set
\begin{equation}\label{eq:cQ-c}
 c_Q=\frac{N_Q}{Q^2},
 \qquad
 c=\frac3{\pi^2}.
\end{equation}
Let $H\in C_c((0,\infty))$ be real-valued. The two-level correlation measure of the Farey fractions, tested on the positive half-line, is
\begin{equation}\label{eq:RQ-pair}
 \mathcal R_Q^{(2)}(H)
 =\frac1{N_Q}
 \sum_{0\leq i<j\leq N_Q-1}
 H\bigl(N_Q(\gamma_j-\gamma_i)\bigr).
\end{equation}
Because $H$ is supported in $(0,\infty)$, each ordered pair with positive difference occurs exactly once in the sum with $i<j$. Thus \eqref{eq:RQ-pair} is the restriction of the usual two-level correlation measure to positive differences, which is the convention used for the density $g_2$ below.

For $c'>0$, define
\begin{equation}\label{eq:f-H-c}
 f_{H,c'}(x,y)
 =\begin{cases}
 H(c'y/x),&0<x\leq1\text{ and }y\geq1,\\
 0,&\text{otherwise}.
 \end{cases}
\end{equation}
This is a bounded compactly supported Borel function. Indeed, if $\supp(H)\subset[a,B]$ with $0<a<B$, a nonzero value requires
\[
 \frac{c'}{B}\leq x\leq1,
 \qquad 1\leq y\leq\frac{B}{c'}.
\]
The condition $y\geq1$ does not remove any positive primitive vector on the BCZ section, since its transformed height is $j/s\geq1$. Let
\[
 G_{H,c'}(s,t)=\widehat{f_{H,c'}}_+(s,t).
\]

\begin{lemma}\label{lem:return-observable-pair}
For every $Q\in\N$ and $0\leq i<N_Q$,
\begin{equation}\label{eq:return-observable-pair}
 G_{H,c_Q}(z_{Q,i})
 =\sum_{r\geq1}
 H\bigl(N_Q(\gamma_{i+r}-\gamma_i)\bigr),
\end{equation}
where the Farey sequence is periodically extended via $\gamma_{i+N_Q}=\gamma_i+1$.
\end{lemma}

\begin{proof}
The vectors in the defining sum for $G_{H,c_Q}$ represent the positive future returns of the horocycle orbit, weighted by $H$. A future hit at time $\tau$ contributes $H(c_Q\tau)$. By \eqref{eq:rth-future-time},
\[
 c_Q\tau_{i,r}
 =\frac{N_Q}{Q^2}Q^2(\gamma_{i+r}-\gamma_i)
 =N_Q(\gamma_{i+r}-\gamma_i).
\]
Compact support of $H$ and the lower bound $R(s,t)=1/(st)\geq1$ imply that only finitely many future returns can contribute.
\end{proof}

\Needspace{16\baselineskip}
\begin{lemma}\label{lem:pair-wrap-regularity}
Let $H\in C_c((0,\infty))$ be real-valued, then
\begin{equation}\label{eq:pair-periodic-average}
 \mathcal R_Q^{(2)}(H)
 =\frac1{N_Q}\sum_{i=0}^{N_Q-1}G_{H,c_Q}(z_{Q,i}).
\end{equation}
The function $G_{H,c}$ is bounded and Riemann integrable on the BCZ
triangle $\Om$. Moreover,
\begin{equation}\label{eq:pair-uniform-c}
 \sup_{(s,t)\in\Om}|G_{H,c_Q}(s,t)-G_{H,c}(s,t)|\longrightarrow0.
\end{equation}
\end{lemma}

\begin{proof}
Choose $0<a<B$ such that $\supp(H)\subset[a,B]$.
We first compare the two sums in \eqref{eq:pair-periodic-average}.
By \cref{lem:return-observable-pair}, the BCZ orbit average is
\[
 \frac1{N_Q}\sum_{i=0}^{N_Q-1}\sum_{r\geq1}
 H\bigl(N_Q(\gamma_{i+r}-\gamma_i)\bigr),
\]
where $\gamma_{i+N_Q}=\gamma_i+1$. The terms with $i+r<N_Q$
are exactly those in $\mathcal R_Q^{(2)}(H)$, with $j=i+r$.
The remaining terms have $i+r\geq N_Q$, so they compare a Farey
fraction in $[0,1)$ with a fraction in the periodic extension at or
beyond $1$. Since $\gamma_{N_Q-1}=1-1/Q$ and $\gamma_{N_Q}=1$,
\[
 \gamma_{i+r}-\gamma_i\geq\frac1Q
 \qquad\text{whenever }0\leq i<N_Q\text{ and }i+r\geq N_Q.
\]
Their arguments in $H$ are therefore at least $N_Q/Q$.
By \eqref{eq:NQ-asymp}, $N_Q/Q\to\infty$, so these terms vanish
once $N_Q/Q>B$. This proves \eqref{eq:pair-periodic-average}.

We next prove boundedness and Riemann integrability by identifying
which primitive lattice vectors can contribute to $G_{H,c'}$.
Set $L=2B/c$ and suppose that $c'\in[c/2,2c]$.
For $(A,j)\in\Z^2_{\prim,+}$, write
\[
 p_{s,t}(A,j)=(x,y)=\left(sA+tj,\frac{j}{s}\right).
\]
Its contribution is $H(c'y/x)$ when $0<x\leq1$.
The condition $y\geq1$ in the definition of $f_{H,c'}$ is automatic,
since $j\geq1$ and $s\leq1$. If this contribution is nonzero, then
\[
 \frac{y}{x}\leq\frac{B}{c'}\leq L,
 \qquad
 1\leq j\leq L,
 \qquad
 s\geq\frac1L.
\]
Indeed, $y\leq Lx\leq L$ and $y=j/s$. Also, since $0<t\leq1$,
\[
 |A|=\frac{|x-tj|}{s}\leq L(1+L).
\]
Thus, for every $(s,t)\in\Om$ and every $c'\in[c/2,2c]$,
only vectors in the same finite set
\[
 W=
 \bigl\{(A,j)\in\Z^2_{\prim,+}:
       1\leq j\leq L,\ |A|\leq L(1+L)\bigr\}
\]
can contribute to $G_{H,c'}$. For all $(s,t)\in\Om$ and
$c'\in[c/2,2c]$, we therefore have
\[
 |G_{H,c'}(s,t)|\leq \#W\,\|H\|_\infty.
\]
For each of these vectors, its contribution is continuous away from
the line segments $sA+tj=0$ and $sA+tj=1$ in $\Om$.
There is no discontinuity arising from the support of $H$, because
$H$ is continuous and vanishes outside $[a,B]$.
These line segments in the $(s,t)$-plane are precisely the loci where
the transformed vector $p_{s,t}(A,j)$, for $(A,j)\in W$, has first
coordinate $0$ or $1$. Thus the discontinuities of $G_{H,c}$ lie
in a finite union of line segments. Define $G_{H,c}=0$ on $\overline\Om\setminus\Om$,
preserving its values at all points already in $\Om$. This extension
adds possible discontinuities only on the boundary of the closed triangle. The resulting bounded function has a discontinuity set of
Lebesgue measure zero, so it is Riemann integrable.

Finally, we compare $G_{H,c_Q}$ and $G_{H,c}$ at the same point
$(s,t)\in\Om$. Each primitive vector with $0<x\leq1$ corresponds
to a positive return time $\tau=y/x$ of the horocycle orbit, and its
contributions to the two transforms are $H(c_Q\tau)$ and $H(c\tau)$.
For sufficiently large $Q$, we have $c_Q\in[c/2,2c]$.
If either contribution is nonzero, then $\tau\leq L$.
Since the first-return time is $R(s,t)=1/(st)\geq1$,
the first positive return occurs at time at least one and consecutive
returns are at least one unit apart. There are therefore at most
$\lceil L\rceil$ return times that can contribute to the difference,
uniformly in $(s,t)$.

Extend $H$ by zero to $[0,\infty)$ and denote its modulus of
continuity by
\[
 \omega_H(\delta)
 =\sup\{|H(u)-H(v)|:u,v\geq0,\ |u-v|\leq\delta\}.
\]
For each return time $\tau\leq L$,
\[
 |H(c_Q\tau)-H(c\tau)|
 \leq\omega_H\bigl(L|c_Q-c|\bigr).
\]
Summing over these return times gives
\[
 \sup_{(s,t)\in\Om}|G_{H,c_Q}(s,t)-G_{H,c}(s,t)|
 \leq\lceil L\rceil\,
       \omega_H\bigl(L|c_Q-c|\bigr).
\]
The right-hand side tends to zero because $c_Q\to c$ and $H$ is
uniformly continuous. This proves \eqref{eq:pair-uniform-c}.
\end{proof}

We are now ready to compute the limiting density. 

\Needspace{12\baselineskip}
\begin{theorem}[Boca--Zaharescu \cite{BocaZaharescu2005}]\label{thm:farey-pair}
For every real-valued $H\in C_c((0,\infty))$,
\begin{equation}\label{eq:pair-limit}
 \lim_{Q\to\infty}\mathcal R_Q^{(2)}(H)
 =\int_0^\infty H(\lambda)g_2(\lambda)\,d\lambda,
\end{equation}
where
\begin{equation}\label{eq:g2}
 g_2(\lambda)
 =\frac{6}{\pi^2\lambda^2}
 \sum_{1\leq j<\pi^2\lambda/3}
 \phi(j)
 \log\left(\frac{\pi^2\lambda}{3j}\right).
\end{equation}

\end{theorem}

\begin{proof}
By \cref{lem:pair-wrap-regularity}, $G_{H,c}$ is bounded and its discontinuity set has $m$-measure zero. Thus the weak convergence $\nu_Q\to m$ from \cite[Theorem~1.3]{AthreyaCheung2014} implies convergence of its integrals against $G_{H,c}$. Moreover, for all sufficiently large $Q$,
\[
 \left|\mathcal R_Q^{(2)}(H)-\int_\Om G_{H,c}\,dm\right|
 \leq\|G_{H,c_Q}-G_{H,c}\|_\infty
 +\left|\int_\Om G_{H,c}\,d\nu_Q-\int_\Om G_{H,c}\,dm\right|.
\]
Both terms tend to zero, so
\begin{equation}\label{eq:pair-to-section}
 \lim_{Q\to\infty}\mathcal R_Q^{(2)}(H)
 =\int_\Om G_{H,c}\,dm.
\end{equation}
Now apply the modular first moment formula \eqref{eq:modular-first} to $f_{H,c}$:
\begin{align}\label{eq:pair-first-apply}
 \int_\Om G_{H,c}\,dm
 &=2\sum_{j\geq1}\phi(j)
 \int_j^\infty\int_0^1
 H\left(c\frac{y}{x}\right)\frac{1}{y^2}\,dx\,dy.
\end{align}
For fixed $y$, make the change of variables
\[
 \lambda=c\frac{y}{x},
 \qquad
 x=\frac{cy}{\lambda},
 \qquad
 |dx|=\frac{cy}{\lambda^2}d\lambda.
\]
As $x$ increases from $0$ to $1$, the variable $\lambda$ decreases from $+\infty$ to $cy$. Hence
\begin{equation}\label{eq:pair-inner-change}
 \int_0^1H\left(c\frac{y}{x}\right)dx
 =cy\int_{cy}^\infty H(\lambda)\frac{d\lambda}{\lambda^2}.
\end{equation}
Substitution into \eqref{eq:pair-first-apply} gives
\[
 \int_\Om G_{H,c}\,dm
 =2c\sum_{j\geq1}\phi(j)
 \int_j^\infty\int_{cy}^\infty
 H(\lambda)\frac{1}{y\lambda^2}\,d\lambda\,dy.
\]
Compact support of $H$ justifies changing the order of integration. For fixed $\lambda$, the $y$-range is $j\leq y\leq\lambda/c$, which has positive length exactly when $j<\lambda/c$. Therefore
\begin{align*}
 \int_\Om G_{H,c}\,dm
 &=2c\int_0^\infty\frac{H(\lambda)}{\lambda^2}
 \sum_{1\leq j<\lambda/c}\phi(j)
 \left(\int_j^{\lambda/c}\frac{dy}{y}\right)d\lambda\\
 &=2c\int_0^\infty\frac{H(\lambda)}{\lambda^2}
 \sum_{1\leq j<\lambda/c}\phi(j)
 \log\left(\frac{\lambda}{cj}\right)d\lambda.
\end{align*}
Finally, $c=3/\pi^2$, so $2c=6/\pi^2$ and $\lambda/c=\pi^2\lambda/3$. This is \eqref{eq:g2}.
\end{proof}

\begin{remark}[The interval where the density vanishes]
The sum in \eqref{eq:g2} is empty for $0<\lambda\leq3/\pi^2$. Hence
\[
 g_2(\lambda)=0,
 \qquad 0<\lambda\leq\frac3{\pi^2}.
\]
In this proof, the vanishing follows from $y\geq j\geq1$ and $x\leq1$, which imply $cy/x\geq c=3/\pi^2$.
\end{remark}

\end{document}